\documentclass{amsart}
\usepackage{amssymb
,amsthm
,amsmath
,amscd
,mathtools
}

\usepackage[pagewise]{lineno}

\usepackage{enumerate}
\usepackage{hyperref}
\usepackage[all]{xy}
\usepackage[top=30truemm,bottom=30truemm,left=25truemm,right=25truemm]{geometry}

\newcommand{\Sp}{\mathrm{Sp}}
\newcommand{\SL}{\mathrm{SL}}
\newcommand{\GL}{\mathrm{GL}}

\newcommand{\Oo}{\mathrm{O}}
\newcommand{\SO}{\mathrm{SO}}
\newcommand{\GSp}{\mathrm{GSp}}
\newcommand{\GSO}{\mathrm{GSO}}
\newcommand{\PGL}{\mathrm{PGL}}
\newcommand{\PGSp}{\mathrm{PGSp}}
\newcommand{\GO}{\mathrm{GO}}

\newcommand{\Hom}{\mathrm{Hom}}
\newcommand{\Gal}{\mathrm{Gal}}
\newcommand{\Ext}{\mathrm{Ext}}

\newtheorem{thm}{Theorem}[section]

\newtheorem{lem}[thm]{Lemma}
\newtheorem{prop}[thm]{Proposition}

\theoremstyle{remark}
\newtheorem{rem}[thm]{Remark}
\newtheorem{ex}[thm]{Example}

\theoremstyle{definition}
\newtheorem{defn}[thm]{Definition}

\numberwithin{equation}{section}

\makeatletter
\def\iddots{\mathinner{\mkern1mu\raise\p@
	\hbox{.}\mkern2mu\raise4\p@\hbox{.}\mkern2mu
	\raise7\p@\vbox{\kern7\p@\hbox{.}}\mkern1mu}}
\def\adots{\mathinner{\mkern2mu\raise\p@\hbox{.}
 \mkern2mu\raise4\p@\hbox{.}\mkern1mu
 \raise7\p@\vbox{\kern7\p@\hbox{.}}\mkern1mu}}
\makeatother

\allowdisplaybreaks
\title{Theta Correspondence and the Prasad conjecture for SL(2)}
\author{Hengfei LU}
\address{School of Mathematics,	Tata Institute of Fundamental Research, Dr. Homi Bhabha Road,	Colaba, Mumbai 400005, INDIA}
\email{hengfei@math.tifr.res.in}
\begin{document}
\maketitle

\begin{abstract}
We use relations between the base change representations and theta lifts, to give a new proof to the local period problems of $\SL(2)$ over a  nonarchimedean quadratic field  extension $E/F$. Then we will verify the Prasad conjecture for $\SL(2)$.  With a similar strategy, we obtain a certain result for the Prasad conjecture for $\Sp(4)$.
\end{abstract}
	\subsection*{Keywords} theta lifts, periods, base change, Prasad's conjecture 
\subsection*{MSC(2000)} 11F27$\cdot$11F70$\cdot$22E50
\tableofcontents
\section{Introduction}
Assume that $F$ is a nonarchimedean local field with characteristic $0$.
Let $G$ be a connected reductive group defined over $F$ and $H$
be a closed subgroup of $G$. Given a smooth irreducible representation $\pi$ of $G(F),$ one may consider  the complex vector space $\Hom_{H(F)}(\pi,\mathbb{C}).$ If it is nonzero, then we say that $\pi$ is  $H(F)$-distinguished, or has a nonzero $H(F)$-period. 
\par
Period problems,
which are closely related to Harmonic Analysis, have been extensively studied for classical groups. 
 The most general situations have been studied  in  \cite{sakellaridis2012periods} when $G$ is split.
 Given a spherical variety $X=H\backslash G$, Sakellaridis and Venkatesh \cite{sakellaridis2012periods}  introduce a certain complex reductive group $\hat{G}_X$ associated with the variety $X$, to deal with  the spectral decomposition of $L^2(H\backslash G)$ under the assumption that $G$ is split.
 In a similar way, Prasad \cite[\S9]{prasad2015arelative} introduces a certain quasi-split reductive group $G^{op}$ to deal with the  period problem when the subgroup $H$ is the Galois fixed points of $G$, i.e. $H=G^{\Gal(E/F)}$, where $E$ is a quadratic field extension of $F$.  In this paper,
  we will mainly focus on the case $G=R_{E/F}\SL_2$ and $H=\SL_2$, where $R_{E/F}$ denotes the Weil restriction of scalars, i.e. the Prasad conjecture \cite[Conjecture 2]{prasad2015arelative} for $\SL_2$.
  
\par
Let $W_F$ (resp. $W_E$) be the Weil group of $F$ (resp. $E$) and  $WD_F$ (resp. $WD_E$) be the Weil-Deligne group. Let $\psi$ be any additive character of $F$ and $\psi_E=\psi\circ tr_{E/F}.$
Assume that $\tau$ is an irreducible smooth representation of $\SL_2(F),$
 with a Langlands parameter 
$\phi_\tau:WD_F\rightarrow \PGL_2(\mathbb{C})$ and a character $\lambda$ of the component group $S_{\phi_\tau }=C(\phi_\tau)/C^\circ(\phi_\tau),$ where $C(\phi_{\tau})$ is the centralizer of $\phi_{\tau}$ in $\PGL_2(\mathbb{C})$ and $C^\circ(\phi_{\tau})$ is the connected component of $C(\phi).$
 Then $\phi_\tau|_{WD_E}$
gives a Langlands parameter of $\SL_2(E)$. The map $\phi_{\tau}\rightarrow \phi_{\tau}|_{WD_E}$ is called the base change map. Prasad's conjecture for $\SL(2)$ predicts the following result, which was shown in \cite{anandavardhanan2003distinguished}.
\begin{thm}\label{localmain}
Let  $E$ be a quadratic field extension of a nonarchimedean local field $F$ with associated 
 Galois group $\Gal(E/F)=\{1,\sigma \}$ and  associated quadratic character $\omega_{E/F}$ of $F^\times.$
Assume that $\tau$ is an irreducible smooth admissible representation of $\SL_2(E)$ with  central character $\omega_{\tau}$ satisfying $\omega_\tau(-1)=1$.
Then the following  are equivalent:
\begin{enumerate}[(i)]
\item  $\tau$ is  $\SL_2(F)$-distinguished;
\item  $\phi_\tau=\phi_{\tau'}|_{WD_E}$ for some irreducible representation $\tau'$ of $\SL_2(F)$ and
$\tau$ has a Whittaker model with respect to a non-trivial  additive character of $E$ which is trivial on $F$. 
\end{enumerate}
\end{thm}
\par
Anandavardhanan and Prasad \cite{anandavardhanan2003distinguished}  deal with the cases for the principal series and square-integrable representations separately,  using the restriction of $\GL_2(F)$-distinguished representations of $\GL_2(E)$. There is a key lemma \cite[Lemma 3.1]{anandavardhanan2003distinguished} that if $\tau$ is $\SL_2(F)$-distinguished, then $\tau$ has a Whittaker model with respect to a non-trivial additive character of $E$ which is trivial on $F$. Moreover, the multiplicity $\dim\Hom_{\SL_2(F)}(\tau,\mathbb{C})$ is invariant under the $\GL_2(F)$-conjugation action on $\tau$. In \cite{anandavardhanan2016distinguished}, they use a similar idea to deal with the case for $\SL_n$, involving the restriction of $\GL_n(F)$-distinguished representations of $\GL_n(E)$.
 In this paper, we will use the local theta correspondence to give a new proof for a tempered representation of $\SL_2(E)$. Then we use Mackey Theory and the double coset decomposition  to deal with the principal series, instead of involving representations of $\GL(2)$.  In order to verify Prasad's conjecture \cite[Conjecture 2]{prasad2015arelative} for $\SL(2),$  we will list all possible explicit parameter lifts $\tilde{\phi}:WD_F\rightarrow \PGL_2(\mathbb{C})$ such that $\tilde{\phi}|_{WD_E}=\phi_{\tau},$ which is different from Prasad's descriptions in \cite[\S 18]{prasad2015arelative}. Our methods can also be used for the $\Sp(4)$-distinction problems over a quadratic field extension, see Theorem \ref{Sp(4)-period}.
\begin{thm}\label{prasadsl(2)}
	Assume that $\tau$ is an irreducible  $\SL_2(F)$-distinguished representation of $\SL_2(E),$ with an enhanced $L$-parameter $(\phi_\tau,\lambda )$, where $\lambda$ is a character of the component group $S_{\phi_{\tau}}$, then \[\dim_\mathbb{C} \Hom_{\SL_2(F)}(\tau,\mathbb{C} )= |F(\phi_\tau)|, \]
	where $F(\phi_\tau )=\{\tilde{\phi}:WD_F\rightarrow \PGL_2(\mathbb{C} ):\tilde{\phi}|_{WD_E}=\phi_{\tau}\mbox{ and }\lambda|_{S_{\tilde{\phi}}}\supset \mathbf{1} \}$ and $|F(\phi_{\tau})|$ denotes its cardinality.
\end{thm}
\begin{rem}
	The statement in Theorem \ref{prasadsl(2)} is slightly different from the original Prasad conjecture for $\SL(2)$. We have used the fact that the degree of the base change map 
	\[\Phi:\Hom(WD_F,\PGL_2(\mathbb{C}))\rightarrow \Hom(WD_E,\PGL_2(\mathbb{C})) \]
at	each parameter $\tilde{\phi}$ is  equal to the size of the cokernel
	\[\mathrm{coker}\{S_{\tilde{\phi}}\rightarrow S_{\phi_{\tau}}^{\Gal(E/F)} \} \]
	for $\tilde{\phi}\in F(\phi_{\tau})$ when $G=\SL(2)$, which is easy to check, see \cite[\S 18]{prasad2015arelative}.
\end{rem}
\begin{rem}
	In \cite[Theorem 1]{beuzart2017distinguished}, Raphael Beuzart-Plessis uses the relative trace formula to give an identity for the multiplicity $\dim_\mathbb{C}\Hom_{H'(F)}(\pi',\chi_{H'})$, where $H'$ is an inner form of $H$ defined over $F$, $\chi_{H'}$ is a quadratic character of $H'(F)$ and $\pi'$ is a stable square-integrable representation of $(R_{E/F}H')(F)=H'(E)$. For example, $H'=\SL_1(D)$ and $H'(E)=\SL_2(E)$, where $D$ is a quaternion division algebra defined over $F$. We are trying to use the local theta correspondence to deal with the distinction problems for the pair $(\SL_2(E),\SL_1(D))$ in the following paper \cite{lu2018SL(D)} as well. More precisely, we will figure out the multiplicity $\dim_\mathbb{C}\Hom_{\SL_1(D)}(\tau,\mathbb{C}) $  for  a smooth irreducible
representation $\tau$ of $\SL_2(E)$.
\end{rem}
\begin{rem}
	In \cite{anan2006integral,anan2013localglobal}, Anandavardhanan and Prasad discuss the global period problems for $\SL_2$ over a quadratic number field extension $\mathbb{E}/\mathbb{F}$. More general, there are several results for the global period problems of $\SL_1(D)$ in \cite[\S9]{anan2013localglobal}, where $\SL_1(D)$ is an inner form of $\SL_2$ defined over a number field $\mathbb{F}$. We hope that we can also use the global theta correspondence to revisit these questions in future.
\end{rem}
Now we briefly describe the contents and the organization of this paper.
In \S $2$, we set up the notation about the  local theta lifts. 
 In \S $3$, we give the proof of Theorem \ref{localmain}, and then we verify Prasad's conjecture for $\SL(2),$ i.e. Theorem \ref{prasadsl(2)} in \S 4. Finally, we give a partial result for the Prasad conjecture for $\Sp_4$, i.e. Theorem \ref{Sp(4)-period}.

\subsection*{Acknowledgments} The author thanks  Wee Teck GAN for his guidance and numerous discussions when he was doing his Ph.D. study at National University of Singapore. He would like to thank Dipendra Prasad for useful discussions as well. He also thanks the anonymous referees for the careful reading and helpful comments, especially for pointing out the inaccurate statement in Theorem \ref{localmain} in the earlier version. 
\section{The Local Theta Correspondences }
In this section, we will briefly recall some results about the local theta correspondence, following \cite{kudla1996notes}.  

Let $F$ be a local field of characteristic zero.
Consider the dual pair $\Oo(V)\times \Sp(W).$
For simplicity, we may assume that $\dim V$ is even. Fix a nontrivial additive character $\psi$ of $F.$
Let $\omega_\psi$ be the Weil representation for $\Oo(V)\times \Sp(W),$ which can be described as follows.
Fix a Witt decomposition $W=X\oplus Y$ and let $P(Y)=\GL(Y)N(Y)$ be the parabolic subgroup stabilizing the maximal isotropic subspace $Y.$ Then \[N(Y)=\{b\in \Hom(X,Y)|~b^t=b \} ,\]
where $b^t\in \Hom(Y^\ast,X^\ast)\cong \Hom(X,Y).$ The Weil representation $\omega_\psi$
can be realized on the Schwartz space $S(X\otimes V)$ and the action of $P(Y)\times \Oo(V)$ is given by the usual formula
$$\begin{cases}
	\omega_{\psi}(h)\phi(x)=\phi(h^{-1}x),&\mbox{for}~ h\in \Oo(V),\\
\omega_{\psi}(a)\phi(x)=\chi_V(\det_Y(a))|\det_Y a|^{\frac{1}{2}\dim V }\phi(a^{-1}\cdot x),&\mbox{for} ~a\in \GL(Y),\\
\omega_\psi(b)\phi(x)=\psi(\langle bx,x\rangle)\phi(x),&\mbox{for}~ b\in N(Y),
\end{cases}$$
where $\chi_V$ is the quadratic character associated to disc$ V\in F^\times/{F^\times}^2$ and
$\langle-,-\rangle$ is the natural symplectic form on $W\otimes V.$
To describe the full action of $\Sp(W),$ one needs to specify the action of a Weyl group element, which acts by a Fourier transform.

If $\pi$ is an irreducible representation of $\Oo(V)$ (resp. $\Sp(W)$), the maximal $\pi$-isotypic quotient has the form 
\[\pi\boxtimes\Theta_\psi(\pi) \]
for some smooth representation of $\Sp(W)$ (resp. $\Oo(V)$). We call $\Theta_\psi(\pi )$
the big theta lift of $\pi.$ It is known that $\Theta_\psi(\pi)$ is of finite length and hence is admissible. Let $\theta_\psi(\pi)$ be the maximal semisimple quotient of $\Theta_\psi(\pi),$ which is called the small theta lift of $\pi.$ Then there is a conjecture of Howe which states that
\begin{itemize}
	\item $\theta_\psi(\pi)$ is irreducible whenever $\Theta_\psi(\pi)$ is non-zero.
	\item the map $\pi\mapsto \theta_\psi(\pi)$ is injective on its domain.
\end{itemize}
This has been proved by Waldspurger \cite{waldspurger1990demonstration} when the residual characteristic $p$ of $F$ is is not $2.$ Recently, it has been proved completely in \cite{gan2014howe,gan2014proof}.
\begin{thm}
	 Howe Conjecture holds.
\end{thm}

 \subsection*{First occurence indices for pairs of orthogonal Witt towers} Let $W_n$ be the $2n$-dimensional symplectic vector space with associated symplectic group $\Sp(W_n)$ and consider the two towers of orthogonal groups attached to the quadratic spaces with nontrivial discriminant. More precisely, let
 $V_E$ (resp. $\epsilon V_E$) be a $2$-dimensional quadratic space with discriminant $E$ and Hasse invariant $+1$ (resp. $-1$), $\mathbb{H}$ be the $2$-dimensional hyperbolic quadratic space over $F$,
\[V_r^+=V_E\oplus \mathbb{H}^{r-1}\quad\mbox{and}\quad V_r^-=\epsilon V_E\oplus\mathbb{H}^{r-1} \]
and denote the orthogonal groups by $\Oo(V_r^+)$ and $\Oo(V_r^-)$ respectively. For an irreducible representation $\pi$ of $\Sp(W_n),$ one may consider the theta lifts $\theta^+_r(\pi)$ and $\theta^-_r(\pi)$ to
$\Oo(V^+_r)$ and $\Oo(V_r^-)$ respectively, with respect to a fixed non-trivial additive character $\psi.$ Set
\[\begin{cases}
r^+(\pi)=\inf\{2r:\theta^+_r(\pi)\neq0 \};\\
r^-(\pi)=\inf\{2r:\theta^-_r(\pi)\neq0 \}.
\end{cases} \]
Then Kudla and Rallis \cite{kudla2005first}, B. Sun and C. Zhu \cite{sun2012conservation} showed:
\begin{thm}
	[Conservation Relation] For any irreducible representation $\pi$ of $\Sp(W_n),$ we have
	\[r^+(\pi)+r^-(\pi)=4n+4=4+2\dim W_n. \]
\end{thm}
On the other hand, one may consider the mirror situation, where one fixes an irreducible representation
of $\Oo(V_r^+)$ or $\Oo(V_r^-)$ and consider its theta lifts $\theta_n(\pi)$ to the tower of symplectic group $\Sp(W_n).$ Then with $n(\pi)$ defined in the analogous fashion,  due to \cite[Theorem 1.10]{sun2012conservation}, we have
\[n(\pi)+n(\pi\otimes\det)=\dim V_r^\pm. \]
 \subsection*{See-saw identities}
 Let $(V,q)$ be a quadratic vector space over $E$ of even dimension. Let $V'=Res_{E/F}V$ be the  same space $V$ but now thought of as a vector space over $F$ with a quadratic form
\[q'(v)=\frac{1}{2}tr_{E/F}q(v). \]
If $W_0$ is a symplectic vector space over $F,$ then $W_0\otimes_F E$
is a symplectic vector space over $E.$ Then we have the following isomorphism of symplectic spaces:
\[Res_{E/F }[(W_0\otimes_F E )\otimes_E V ]\cong W_0\otimes V'=\mathbf{W} \]
There is a pair 
\[(\Sp(W_0),\Oo(V') )\mbox{  and  }(\Sp(W_0\otimes E),\Oo(V)) \]
of dual reductive pairs in the symplectic group $\Sp(\mathbf{W}).$
A pair $(G_1,H_1)$ and $(G_2,H_2)$ of dual reductive pairs in a symplectic group is called a see-saw pair if $H_1\subset G_2$ and $H_2\subset G_1.$
\begin{lem}\cite{kudla1984seesaw}
	For a see-saw pair of dual reductive pairs $(G_1,H_1)$ and $(G_2,H_2)$, let $\pi_1$ be an irreducible representation of $H_1$ and $\pi_2$ of $H_2$, then we have the following isomorphism
	\[\Hom_{H_1}(\Theta_\psi(\pi_2),\pi_1 )\cong \Hom_{H_2}(\Theta_\psi(\pi_1),\pi_2 ). \]
	\end{lem}
\subsection*{Quadratic spaces}
Let $K/E$ be a quadratic field extension and  $V=V_K$ be a $2$-dimensional quadratic space over $E$ with the norm map $N_{K/E}.$ Set $\varpi$ to be the uniformizer of $\mathcal{O}_F$ and  $\Gal(K/E)=\langle s\rangle.$ Let $u$ be a unit in $\mathcal{O}_F^\times\setminus{\mathcal{O}_F^\times}^2$. Assume that the Hilbert symbol $(\varpi,u)_F$ is $-1$.
\begin{ex}
\begin{enumerate}[(i)]
Assume that $p$ is odd. Let $L=F(\sqrt{-\varpi})$ be a quadratic field extension over $F$ with associated quadratic character $\omega_{L/F}=\omega_{F(\sqrt{-\varpi})/F}$ by Local Class Field Theory. Let $K$ be a quadratic field extension over $E$, then $V_K$ is a $2$-dimensional quadratic space over $E$ with norm map $N_{K/E}$. We may regard $V_K$ as a $4$-dimensional quadratic space $V'$ over $F$ with quadratic form $q'(k)=\frac{1}{2}tr_{E/F}N_{K/E}(k)$ for $k\in K$.
	\item If $E=F(\sqrt{\varpi} )$ is ramified,  then
	\begin{itemize}
		\item If $K=E(\sqrt{u} )$, then the discriminant $disc(V')=1\in F^\times/{F^\times}^2$ and the Hasse invariant $\epsilon(V')=-1.$
		\item If $K=E(\sqrt[4]{\varpi})$, then $V'=V_{L}\oplus \mathbb{H}$ and $ disc(V')=-\varpi\in F^\times/{F^\times}^2.$
		\item If $K=E(\sqrt[4]{\varpi}\cdot\sqrt{u} )$, then $disc(V')={L}. $
	\end{itemize}
	\item If $E=F(\sqrt{u})$ is unramified, then 
	\begin{itemize}
		\item  If $K=E(\sqrt{\varpi}),$ then $disc(V')=1$ and $\epsilon(V')=-(-1,\varpi)_F=\begin{cases}+1,&\mbox{if }-1\in u{F^\times}^2;\\-1,&\mbox{if }-1\in{F^\times}^2.
		\end{cases}$
		\item If $K=E(\sqrt{u'})$ and $u'\notin F^\times, $ then $disc(V')=N_{E/F}(u')\in F^\times/{F^\times}^2.$
	\end{itemize}
\end{enumerate}
\end{ex}

If $-1\in ({F^\times})^2$ is a square in $F^\times$ and the discriminant of $V'=Res_{E/F}V_K$ is same as the discriminant of the $2$-dimensional vector space $E$ over $F,$ i.e. $disc(V')=E$, then $\chi_{V'}$ is $\omega_{E/F}$ and its special orthogonal group, denoted by $\SO(V')=\SO(3,1),$ is isomorphic to
\[\SO(3,1)=\frac{ \{(g,\lambda)\in \GL_2(E)\times F^\times:\lambda^2 N_{E/F}(\det g)=1 \}}{ \{(t, N_{E/F}(t)^{-1}):t\in E^\times \}}\cong\frac{\{g\in \GL_2(E):\det(g)\in F^\times \}  }{F^\times } .\]
Set $K^1=\{k\in K^\times:k\cdot k^s=1 \},$  then there is a natural embedding 
\[\Oo(V_K)=K^1\rtimes\mu_2 \subset \SO(3,1)  \mbox{  where  }  K^1= \SO(V_K)\subset \GL_2(E) .\]
In general, the discriminant $disc(V')$ may not be equal to $E$. There is a group embedding  $K^1\hookrightarrow \GL_2(L')$ where $L'=F(\delta),\delta^2=N_{E/F}(u')$ if $K=E(\sqrt{u'}).$
\begin{rem}
	If $V'=Res_{E/F}V_K$ has discriminant $1\in F^\times/{F^\times}^2$
	and Hasse invariant $+1$, then $V'$ is called  a split $4$-dimensional quadratic space over $F.$  Set $\SO_{2,2}(F)=\SO(V')$ to be the special orthogonal group.
\end{rem}
\subsection*{Degenerate principal series representations} Let $V_K$ be a $2$-dimensional quadratic space over $E$ with the norm map $N_{K/E}.$ Assume that $V'=Res_{E/F}V_K$ is a split $4$-dimensional quadratic space over $F.$ There is a natural embedding $\Oo(V_K)\hookrightarrow \Oo_{2,2}(F).$ Let $P$ be a Siegel parabolic subgroup of $\Oo_{2,2}(F).$ Assume that $\mathcal{I}(s)$ is the degenerate principal series of $\Oo_{2,2}(F).$
Let us consider the double coset decomposition $P\backslash\Oo_{2,2}(F)/\Oo(V_K).$ 
\begin{itemize}
	\item If $K$ is a field, then there are four open orbits in $P\backslash\Oo_{2,2}(F)/\Oo(V_K)$.
	\item If $K=E\oplus E,$ then there are one closed orbit and three open orbits in $P\backslash\Oo_{2,2}(F)/\Oo_{1,1}(E)$. 
\end{itemize}
Assume that there is a stratification $P\backslash \Oo_{2,2}(F)/\Oo(V_K)=\sqcup_{i=0}^r X_i$ such that  $\sqcup_{i=0}^k X_i$ is open for each $k$ lying in $\{0,1,2,\cdots, r\}.$
Then
there is an $\Oo(V_K)$-equivariant filtration $\{I_i \}_{i=0,1,2,\cdots,r }$ of $\mathcal{I}(s)|_{\Oo(V_K)}$ such that 
$$0=I_{-1}\subset I_0\subset I_1\subset \cdots\subset I_r=\mathcal{I}(s)|_{\Oo(V_K)}$$
  and the smooth functions in the quotient $ I_{i}/I_{i-1}$ are supported on a single orbit $X_{i}$  in  $P\backslash \Oo_{2,2}(F)/\Oo(V_K)$. 
\begin{defn} Given an irreducible representation $\pi$ of $\Oo(V_K),$ if $\Hom_{\Oo(V_K)}(I_{i+1}/I_i,\pi )\neq0$ implies that  $I_{i+1}/I_{i}$ is supported on the open orbits in $P\backslash \Oo_{2,2}(F)/\Oo(V_K),$ then we say that the representation $\pi$ does not occur on the boundary of $\mathcal{I}(s).$
	\end{defn}
It is well-known that only the open orbits can support supercuspidal representations. Due to the Casselman criterion for a tempered representation,  only the open orbits  can support the tempered representations  in our case if $s=\frac{1}{2}$, see \cite[Lemma 4.2.9]{lu2016new}. 
\section{Proof of  Theorem \ref{localmain}}
Before we prove Theorem \ref{localmain},
let us recall some facts. 
\begin{lem}
	If the discriminant of $V'=Res_{E/F}V_K$ is $E,$ then the theta lift of the trivial representation from $\SL_2(F)$ to $\SO(3,1)=\SO(V')$ is a character, i.e.
	\[\Theta_\psi(\mathbf{1})=\mathbf{1}\boxtimes\omega_{E/F}. \]
\end{lem}
\begin{proof}
	Due to \cite[Theorem 2.4.11]{lu2016new}, the big theta lift of the Steinberg representation $St$ from $\GL_2^+(F)$ to $\GSO(3,1)$ is $\Theta_\psi(St)=St_E\boxtimes\omega_{E/F}.$ By a similar argument, one can get $\Theta_\psi(\mathbf{1})=\mathbf{1}\boxtimes\omega_{E/F}. $
	Notice that
	 $$\Theta_\psi(\mathbf{1}|_{\SL_2})=\Theta_\psi(\mathbf{1} )|_{\SO(3,1)} ,$$
	then  we are done.
\end{proof}
\begin{rem}
	In fact, the  theta lift $\theta'_\psi(\mathbf{1})$ from $\SL_2(F)$ to $\Oo(3,1)$
	remains irreducible when restricted to $\SO(3,1),$ see \cite[\S 5]{prasad1993local}.
\end{rem}
Now we start to prove Theorem \ref{localmain}.
\subsection*{Proof of Theorem \ref{localmain}}
According to the representation $\tau,$ we separate the proof  into four cases:
\begin{itemize}
	\item $\tau$ is a supercuspidal representation, see (A);
	\item  $\tau$ is an irreducible principal series representation, see (B);
	\item $\tau$ is a Steinberg representation $St_E$, see (C);
	\item $\tau$ is a constituent of a reducible principle series $I(\chi)$ with $\chi^2=1,$ see (D).
\end{itemize}
These exhause all irreducible smooth representations of $\SL_2(E)$.
\begin{enumerate}[(A).]
	\item 
If $\tau$ is supercuspidal, then 
there exists a character $\mu:K^\times\rightarrow\mathbb{C}^\times$
such that $\phi_\tau=i\circ (Ind_{W_K}^{W_E}\mu)$, where 
\begin{itemize}
	\item  $W_K$ is the Weil group of $K$, where $K$ is a quadratic field extension over $E$;
	\item $\mu$
	does not factor through the norm map $N_{K/E},$ so
	 the irreducible Langlands parameter $$Ind_{W_K}^{W_E}\mu:W_E\rightarrow\GL_2(\mathbb{C})$$ corresponds to a dihedral supercuspidal representation of $\GL_2(E)$ with respect to $K$;
	 \item $i:\GL_2(\mathbb{C})\rightarrow \PGL_2(\mathbb{C})$ is the projection map, which coincides with the adjoint map $$Ad:\GL(2)\rightarrow\SO(3).$$
\end{itemize}
 In fact, the Langlands parameter $\phi$ of the representation $\Sigma$ of $\Oo(V_K)$, where $\tau=\theta_{\psi }(\Sigma)$, is given by
 \[\phi(g)=\begin{cases}
 \begin{pmatrix}
 \chi_K(g)\\&\chi_K^{-1}(g)
 \end{pmatrix}&\mbox{if } g\in W_K\\
 \begin{pmatrix}
 0&1\\1&0
 \end{pmatrix}&\mbox{if }g=s
 \end{cases} \]
 where $s\in W_E\setminus W_K$ and the character $\chi_K:W_K\rightarrow\mathbb{C}^\times$ is the pull back of a nontrivial character $\mu_1$ of $K^1$ under the  map $K^\times\rightarrow K^1$ via $k\mapsto k^sk^{-1}$, i.e. $\chi_K(k)=\mu_1(k^sk^{-1})$, see \cite[\S6.4]{kudla1996notes}. Furthermore, there is an isomorphism between two Langlands parameters of $\Oo(2)$
 \[\phi\otimes\omega_{K/E}\cong Ind_{W_K}^{W_E}\frac{\mu^s}{\mu}. \]
 In other words, one has $\chi_K=\mu^s\mu^{-1}$ and $\mu_1=\mu|_{K^1}$ is the restricted character.

Moreover, if $\mu_1^2\neq\mathbf{1}$, then $\tau=\theta_\psi(Ind_{\SO(V_K)}^{\Oo(V_K)}(\mu_1) ).$
If $\mu_1^2=\mathbf{1}$, then there are two extensions of $\mu_1$ from $\SO(V_K)$ to $\Oo(V_K)$, denoted by $\mu_1^\pm$.  For convenience, if $\mu_1^2\neq\mathbf{1}$, we denote the irreducible representation $Ind_{\SO(V_K)}^{\Oo(V_K)}(\mu_1)$ by $\mu_1^+$ as well.
Assume that
$\tau=\Theta_\psi(\mu_1^+ )$ is supercuspidal.
\par
If the discriminant $disc V'=L\in F^\times/(F^\times )^2$ is nontrivial,  by the see-saw diagram
\[\xymatrix{\tau& \SL_2(E)\ar@{-}[rd] & \Oo(V')\ar@{-}[ld] &\Theta_\psi(\mathbf{1} )\\ \mathbf{1}&\SL_2(F)& \Oo(V_K)&{\mu_1}^+ } \]
 one has an isomorphism
\[\Hom_{\SL_2(F)}(\tau,\mathbb{C} )\cong \Hom_{\Oo(V_K)}(\mathbf{1}\boxtimes\omega_{L/F},\mu_1^+ ) \]
which is nonzero if and only if $\mu_1=\mathbf{1}.$ But $\Hom_{K^1}(\mathbf{1},\mu_1)=0$, then
 $\Hom_{\SL_2(F)}(\tau,\mathbb{C})=0$.
\par
If the discriminant of $V'$ is $1\in F^\times/(F^\times)^2 $ and its Hasse invariant is $-1,$ then the theta lift $\theta_\psi(\mathbf{1})$ from $\SL_2(F)$ to $\Oo(V')$ is zero by Conservation Relation, so that
\[\Hom_{\SL_2(F) }(\tau,\mathbb{C} )=\Hom_{\Oo(V_K)}(\Theta_\psi(\mathbf{1}),\theta_\psi(\tau) )=0. \]
\par
If $V'\cong\mathbb{H}^2$ is a split $4$-dimensional quadratic space over $F,$  we denote $\mathcal{I}(s)$ the degenerate principal series of $\Oo_{2,2}(F)$ and we assume that $F^\times/(F^\times)^2\supset \{1,u,\varpi,u\varpi \}$ and $E=F(\sqrt{u} )$ with associated Galois group $\Gal(E/F)=\langle\sigma\rangle,$ then
\begin{equation}\label{openorbit}
\Hom_{\SL_2(F) }(\tau,\mathbb{C} )=\Hom_{\Oo(V_K)}(\mathcal{I}(\frac{1}{2} ),\mu_1^+ )\cong \bigoplus_{j=1}^4 \Hom_{\Oo(V_j)}(\mu_1^+,\mathbb{C} ) 
\end{equation} 
where $K=F(\sqrt{\varpi},\sqrt{u} )$ is a biquadratic field over $F$,  $V_1=V_{E'}$ $(E'=F(\sqrt{\varpi})$ is a quadratic field extension over $F)$ is a $2$-dimensional quadratic space over $F$ with quadratic form  $q(e')=N_{E'/F }(e')$, Hasse invariant $+1$ and  quadratic character $$\chi_{V_1}=\omega_{E'/F}=\omega_{F(\sqrt{\omega})/F},$$ $V_2=\epsilon'V_1 (\epsilon'\in F^\times\setminus N_{E'/F}(E')^\times)$
is the $2$-dimensional quadratic space $F(\sqrt{\varpi})$ with  quadratic form $\epsilon'N_{E'/F}$, Hasse invariant $-1$ and  quadratic character $\chi_{V_2}=\chi_{V_1};$
similarly, $V_3=V_{E''}$ is a $2$-dimensional quadratic space over $F$ with quadratic character $\omega_{F(\sqrt{\varpi u})/F}$ and Hasse invariant $+1$, where $E''=F(\sqrt{\varpi u} )$ is a  quadratic field extension over $F$ and $V_4=\epsilon''V_3 $ with Hasse invariant $-1$, where $\epsilon''\in F^\times\setminus N_{E''/F}(E'')^\times.$
In this case,  \eqref{openorbit} can be rewritten as the following identity
\begin{equation}\label{thesum}
\dim_{\mathbb{C}} \Hom_{\SL_2(F)}(\tau,\mathbb{C} )=\sum_{j=1}^4\dim_{\mathbb{C}} \Hom_{\Oo(V_j)}(\mu_1^+,\mathbb{C} ) 
\end{equation}
which is nonzero if and only if one of the following holds:
\begin{itemize}
	\item $\mu(x-y\sqrt{\varpi})=\mu(x+y\sqrt{\varpi})$ for $x,y\in F;$ 
	\item $\mu(x-y\sqrt{u\varpi})=\mu(x+y\sqrt{u\varpi}) $ for $x,y\in F.$ 
\end{itemize}
\begin{rem}
	Because $\mu^s\neq\mu$, these two conditions can not hold at the same time unless $p=2.$
\end{rem}

We would like to highlight the fact about the group embeddings $\Oo(V_j)\hookrightarrow K^1\rtimes <s>$ for $j\in\{1,2\}$.
There is a natural group embedding $\SO(V_1)\rtimes<s>\rightarrow K^1\rtimes <s>$. Via the isomorphism between two quadratic $E$-vector spaces $(V_{E'}\otimes_F E,\epsilon' N_{E'/F})\cong (V_K,N_{K/E}),$ one has an identity
\[\dim\Hom_{\Oo(\epsilon'V_{E'})}(\mu_1^+,\mathbb{C} )=\dim\Hom_{\Oo(V_{E'}) }((\mu_1^+)^{g_{\epsilon'}},\mathbb{C} )  \]
where $(\mu_1^+)^{g_{\epsilon'}}$ is a representation of $\Oo(V_K)$ given by 
\[(\mu_1^+)^{g_{\epsilon'}}(x)=\mu_1^+(g^{-1}_{\epsilon'}xg_{\epsilon'}),x\in\Oo(V_K),g_{\epsilon'}\in \GSO(V_K)=K^\times \mbox{ with }N_{K/E}(g_{\epsilon'})=\epsilon'. \]
Moreover, if the Whittaker datum is fixed, then the enhanced $L$-parameter of $(\mu_1^+)^{g_{\epsilon'}}$ is known if the enhanced $L$-parameter of $\mu_1^+$ is given, see \cite[\S3.6]{atobe2017evenorth}.
	\subsection*{Assume $p\neq2$}
\begin{enumerate}[(i).]
	\item If $\mu_1^2\neq\mathbf{1}$,  
	then $Ind_{\SO(V_K)}^{\Oo(V_K)}(\mu_1)$ is irreducible and
	 \[\dim \Hom_{\Oo(V_2)}(Ind_{\SO(V_K)}^{\Oo(V_K)}(\mu_1),\mathbb{C} )=\dim \Hom_{\Oo(V_1)}(Ind_{\SO(V_K)}^{\Oo(V_K)}(\mu_1),\mathbb{C} ). \]
	\item If $\mu_1^2=\mathbf{1}$, then $\mu^2=\chi_E\circ N_{K/E}$ and $\mu^s=-\mu$, so
	\[\dim \Hom_{\Oo(V_2)}(\mu_1^+,\mathbb{C} )=\dim \Hom_{\Oo(V_1) }(\mu_1^- ,\mathbb{C} ).\]
\end{enumerate}
	Hence, if $p\neq2
	$, the equality \eqref{thesum} imply the following: 
	\begin{itemize}
		\item If $\mu_1^2\neq\mathbf{1}$ and $\mu|_{E'}$ factors through the norm map $N_{E'/F}$ for $E'\neq E,$ then $$\dim \Hom_{\SL_2(F)}(\tau,\mathbb{C} )=2.$$
		\item If $\mu_1^2=\mathbf{1}$ and $\mu|_{E'} $ factors through the norm map $N_{E'/F}$ for $E'\neq E,$  then $$\dim \Hom_{\SL_2(F)}(\tau,\mathbb{C} )=1. $$
	\end{itemize}
If $\mu_1^2=\mathbf{1}$  and $\tau=\theta_\psi(\mu_1^+ )$ is $\SL_2(F)$-distinguished, then
\[\dim \Hom_{\SL_2(F)}(\theta_\psi(\mu_1^-),\mathbb{C})=\dim \Hom_{\Oo(V_K) }(\mathcal{I}(\frac{1}{2}), \mu_1^-) \]
which is equal to $$\sum_{j=1}^4\dim \Hom_{\Oo(V_j)}(\mu_1^-,\mathbb{C} )=\dim \Hom_{\SL_2(F) }(\theta_\psi(\mu_1^+),\mathbb{C} ). $$
Hence the dimension $$\dim \Hom_{\SL_2(F)}(\theta_\psi(\mu_1^-),\mathbb{C} )=1  $$ if and only if $\mu|_{E'}$ factors through the norm map $N_{E'/F}$ for $E'\neq E.$
\par
\subsection*{Assume $p=2$}
\begin{enumerate}[(i).]
	\item Suppose that there are two distinct quadratic fields $E'$ and $E''$ over $F$ such that $\mu|_{E'}=\chi_F'\circ N_{E'/F}$ and $\mu|_{E''}=\chi_F''\circ N_{E''/F}$. Furthermore,  $\frac{\chi_F'}{\chi_F''}$ is a quadratic character of $F^\times$ that is not trivial restricted on the Weil group $W_K$ of $K$, i.e. $\frac{\chi_F'}{\chi_F''}$ is different from three quadratic characters $\omega_{E/F}$, $\omega_{E'/F}$ and $\omega_{E''/F}$, which may happen only when $p=2.$ In this case, $\mu^s(t)=\mu(t)\frac{\chi_F'}{\chi_F''}(t)$ for $t\in W_K$,
	\[\dim \Hom_{\Oo(V_1)}(\mu_1^+,\mathbb{C} )=\dim\Hom_{\Oo(V_2)}(\mu_1^+,\mathbb{C})\]  and   $\dim \Hom_{\SL_2(F)}(\tau,\mathbb{C})=4$ by the identity \eqref{thesum}. 
	\item Given a cuspidal representation $\pi$ of $\GL_2(E)$ with $ \pi|_{\SL_2(E)}\supset\tau$, if $\pi$ is not dihedral with respect to any quadratic extension $K$ over $E$, then
	$\pi|_{\SL_2(E)}=\tau$ is irreducible. 
	\par
	We consider a $4$-dimensional quadratic space $X$ over $F$ with discriminant $E,$ then the orthogonal group $\Oo(X)=\Oo(3,1)$ can be naturally embedded into the orthogonal group $\Oo(X\otimes_F E)=\Oo(2,2)(E)$. Let $\pi\boxtimes\pi$ be the irreducible representation of the similitude special orthogonal group $\GSO(2,2)(E)$. By the property of the big theta lift $\Theta(\pi)$ from $\GL_2(E)$ to $\GSO(2,2)(E)$,  $$(\pi\boxtimes\pi)|_{\SO(2,2)(E)}=\Theta(\pi)|_{\SO(2,2)(E)}=\Theta(\pi|_{\SL_2(E)})=\Theta(\tau)$$ is irreducible since $\tau$ is supercuspidal. Let $$\mathfrak{I}(s,\omega_{E/F})=\{f:\Sp_4(F)\longrightarrow\mathbb{C}|f(mng)=|\det(m)|^{s+3/2}\omega_{E/F}( m)f(g),m\in\GL_2(F),g\in\Sp_4(F) \}$$ be the degenerate principal series of $\Sp_4(F)$. Assume that
	$(\pi\boxtimes\pi)^+$ is the unique extension from $\GSO(2,2)(E)$ to $\GO(2,2)(E)$ which participates with the theta correspondence with $\GL_2(E)$. Then $(\pi\boxtimes\pi)^+|_{\Oo(2,2)(E)}$ is irreducible. Considering  the following see-saw diagram
	\[\xymatrix{\mathfrak{I}(\frac{1}{2},\omega_{E/F})&\Sp_4(F)\ar@{-}[rd] & \Oo(2,2)(E)&(\pi\boxtimes\pi
		)^+\\\pi|_{\SL_2(E)} &\SL_2(E)\ar@{-}[ru] &\Oo(3,1)(F)&\mathbb{C} } \]
	due to the structure of $\mathfrak{I}(\frac{1}{2},\omega_{E/F})$  in
	\cite[Proposition 7.2]{gan2014formal}, one can get an equality
	\[\dim \Hom_{\SL_2(E)}(\mathfrak{I}(\frac{1}{2},\omega_{E/F}),\pi )=\dim \Hom_{\Oo(3,1)(F) }((\pi\boxtimes\pi
	)^+,\mathbb{C}). \]
	The supercuspidal representation $\pi|_{\SL_2(E)}$ does not occur on the boundary of $\mathfrak{I}(\frac{1}{2},\omega_{E/F}),$ then
	\[\dim \Hom_{\SL_2(E) }(\mathfrak{I}(\frac{1}{2},\omega_{E/F}),\pi )=\dim \Hom_{\SL_2(F)}(\pi^\vee,\mathbb{C} ). \]
	By the conservation relation, the fact that the first occurrence indice of the determinant map $\det$ of $\Oo(3,1)(F)$ is $4$ implies that $\Theta_\psi(\det)$ from $\Oo(3,1)(F)$ to $\Sp(W_2)=\Sp_4(F)$ is zero and
	 $$\Hom_{\Oo(3,1)(F) }((\pi\boxtimes\pi
	)^-,\mathbb{C} )\cong \Hom_{\Oo(3,1)(F)}((\pi\boxtimes\pi)^+,\det )=\Hom_{\SL_2(E)}(\Theta_\psi(\det),\pi|_{\SL_2(E)})=0.$$
	Hence 
	\begin{equation}
	\begin{split}
	\dim \Hom_{\SL_2(F)}(\pi^\vee,\mathbb{C} )&=\dim\Hom_{\Oo(3,1)(F) }((\pi\boxtimes\pi)^+,\mathbb{C})\\
	&=\dim\Hom_{\Oo(3,1)(F) }((\pi\boxtimes\pi)^+,\mathbb{C})+\dim\Hom_{\Oo(3,1)(F) }((\pi\boxtimes\pi)^-,\mathbb{C})\\
	&=
	\dim \Hom_{\Oo(3,1)(F)}(Ind_{\SO(2,2)(E)}^{\Oo(2,2)(E)}(\pi\boxtimes\pi)|_{\SO(2,2)(E)},\mathbb{C} )\\
	&=\dim\Hom_{\SO(3,1)(F)}((\pi\boxtimes\pi),\mathbb{C})\\
	&=\dim\Hom_{\GSO(3,1)(F)}(\pi\boxtimes\pi,\mathbb{C} )\\
	&=\dim \Hom_{\GL_2(E)}(\pi^\sigma,\pi^\vee ) .
	\end{split}
	\end{equation}
	Therefore, if $\pi$ is not dihedral with respect to any quadratic field extension $K$ over $E$ and so $\tau=\pi|_{\SL_2(E)}$ is irreducible, then the following are equivalent:
	\begin{itemize}
		\item $\pi^\sigma\cong\pi^\vee$, i.e. $\phi_\pi$ is  conjugate-self-dual in the sense of \cite[\S 3]{gan2011symplectic}; 
		\item $\dim \Hom_{\SL_2(F)}(\tau,\mathbb{C})=1.$
	\end{itemize}
\end{enumerate}
	\begin{rem}
	This method can be used to deal with the case when $\tau$ is the Steinberg representation $St_E$ of $\SL_2(E)$, which will imply  $\dim\Hom_{\SL_2(F)}(St_E,\mathbb{C})=1$ directly. It will appear  in the proof of Theorem \ref{Sp(4)-period} as well.
\end{rem}
	\item 
Let $\chi$ be a unitary character of $E^\times.$
If $\tau=I(z,\chi)=Ind_{B(E)}^{\SL_2(E) }\chi|-|_E^z$ (normalized induction) is an irreducible principal series, by the double coset decomposition for $B(E)\backslash\SL_2(E)/\SL_2(F)$
\[\SL_2(E)=B(E)\SL_2(F)\sqcup B(E)\eta_1 \SL_2(F)\sqcup B(E)\eta_2 \SL_2(F) \]
where $\eta_1=\begin{pmatrix}
1\\ \sqrt{d}&1
\end{pmatrix}$ and $\eta_2=\begin{pmatrix}
1\\\epsilon\sqrt{d}&1
\end{pmatrix},\epsilon\in F^\times\backslash N_{E/F}(E^\times),$ then there is a short exact sequence
\begin{equation}\label{doublecoset}
\xymatrix{ \Hom_{F^\times}(|-|_E^z\chi,\mathbb{C})\ar@{^{(}->}[r]& \Hom_{\SL_2(F)}(\tau,\mathbb{C} )\ar[r]& \prod_{j=1}^2 \Hom_{E^1}(\tau^{\eta_j},\mathbb{C} )\ar[r]&\Ext^1_{F^\times}(|-|_E^z\chi,\mathbb{C} ) }
\end{equation}
where $\tau^{\eta_j}\Big(\begin{pmatrix}
a&\ast\\ &\bar{a}
\end{pmatrix}\Big)=\chi(a)$ for $a\in E^1=\ker \{N_{E/F}:E^\times\rightarrow F^\times \}$. 
Then $\Hom_{\SL_2(F)}(\tau,\mathbb{C})\neq0 $ if and only if one of the following conditions holds:
\begin{itemize}
	\item $\chi|_{F^\times}=\mathbf{1}$ and $z=0;$
	\item $\chi=\chi_F\circ N_{E/F} $.
\end{itemize}
In order to verify the Prasad conjecture, we need to figure out the exact dimension $\dim_\mathbb{C}\Hom_{\SL_2(F)}(\tau,\mathbb{C} )$.
\begin{enumerate}[(i).]
	\item If $\chi$ is trivial and $z=0,$ then $\tau=I(\mathbf{1})$ is irreducible and $\dim \Hom_{\SL_2(F)}(\tau,\mathbb{C} )=2. $
	\item If $\chi=\chi_F\circ N_{E/F} $ with $\chi^2=\mathbf{1}\neq\chi$ and $z=0,$ then $I(\chi)$ is reducible, which belongs to the tempered cases and we will  discuss  later, see (D).
	\item If $\chi=\chi_F\circ N_{E/F} $ with $\chi^2\neq\mathbf{1},$ then 
	$\dim \Hom_{\SL_2(F)}( \tau,\mathbb{C})=2. $
	\item If $\chi$ does not factor through $N_{E/F}$ but $\chi|_{F^\times}=\mathbf{1}$ and $s=0,$ then
	\[\dim \Hom_{\SL_2(F)}(\tau,\mathbb{C} )=1. \]
\end{enumerate}

\item
If $\tau=St_E$ is a Steinberg representation of $\SL_2(E)$, then  the exact sequence \eqref{doublecoset} implies that
\[\dim\Hom_{\SL_2(F)}(I(|-|_E),\mathbb{C})=2,\]  so that  $\dim \Hom_{\SL_2(F) }(St_E,\mathbb{C} )=2-1=1$. 
\item Assume that $\tau$ is tempered.
If $\tau\subset I(\omega_{K/E})$ is an irreducible constituent of a reducible principal series,  set $\chi=\omega_{K/E},~\chi^+(\omega)=1,\omega=\begin{pmatrix}
&1\\1
\end{pmatrix},$ then from \cite[Page 86]{kudla1996notes}, we can see that
\[I(\omega_{K/E} )=\theta_\psi(\chi^+ )\oplus\theta_\psi(\chi^-)\mbox{ where }\chi^-= \chi^+\otimes\det  \]
and $\tau=\theta_\psi(\chi^+)=\Theta_\psi(\chi^+),$ where  $\theta_\psi(\chi^+)$ is the theta lift of $\chi^+$ from
$\Oo_{1,1}(E)$ to $\SL_2(E).$ By the see-saw diagram
\[\xymatrix{\tau& \SL_2(E)\ar@{-}[rd]&\Oo_{2,2}(F)\ar@{-}[ld]&\mathcal{I}(1/2)\\\mathbb{C}&\SL_2(F)&\Oo_{1,1}(E)&\chi^+ } \]
where $\mathcal{I}(s)$ is the principal series of $\Oo_{2,2}(F)$,   we have an identity
\[\dim \Hom_{\SL_2(F) }(\tau,\mathbb{C} )=\dim \Hom_{\Oo_{1,1}(E)}(\mathcal{I}(\frac{1}{2} ),\chi^+) \]
which is equal to $$\dim \Hom_{\Oo_{1,1}(F)}(\chi^+,\mathbb{C} )+\dim \Hom_{\Oo(V_E)}(\chi^+,\mathbb{C} )+\dim \Hom_{\Oo(\epsilon V_E)}(\chi^+,\mathbb{C}) .$$

If $\chi|_{F^\times}=1,$ then $\dim \Hom_{\Oo_{1,1}(F)}(\chi^+,\mathbb{C})=1 $ and $\dim \Hom_{\Oo_{1,1}(F)}(\chi^-,\mathbb{C})=0.$ If $\chi=\chi_F\circ N_{E/F},$ then $\dim \Hom_{\Oo(V_E)}(\chi^+,\mathbb{C} )=1. $
 Hence we have the conclusion:
 \begin{itemize}
 	\item if $\chi=\omega_{K/E}=\chi_F\circ N_{E/F}$ with $\chi_F^2=1,$ then
 	\[\dim\Hom_{\Oo(\epsilon V_E)}(\chi^+,\mathbb{C})=\dim\Hom_{\Oo(V_E)}(\chi^+,\mathbb{C})=1 \]
 	and   $$\dim \Hom_{\SL_2(F)}(\tau,\mathbb{C} )=3;$$
 	\item 
 	if $\chi=\chi_F\circ N_{E/F}$ with $\chi_F^2=\omega_{E/F},$ then
 	$$\dim\Hom_{\Oo(\epsilon V_E)}(\chi^+,\mathbb{C})=\dim\Hom_{\Oo(V_E)}(\chi^-,\mathbb{C}) $$ and $$\dim \Hom_{\SL_2(F)}(\theta_\psi(\chi^+),\mathbb{C} )=\dim\Hom_{E^1}(\chi,\mathbb{C})=1; $$
 	\item if $\chi $ does not factor through the norm map $N_{E/F},$ but $\chi|_{F^\times}=1, $ then
 	\[\dim \Hom_{\SL_2(F)}(\tau,\mathbb{C} )=1. \]
 \end{itemize}
  In this case, if $\dim \Hom_{\SL_2(F)}(\theta_\psi(\chi^+),\mathbb{C} )\neq0 ,$  then $\dim \Hom_{\SL_2(F)}(\theta_\psi(\chi^- ),\mathbb{C} )$ is equal to the sum
  \[\dim \Hom_{\Oo_{1,1}(F)}(\chi^+,\det)+\dim \Hom_{\Oo(V_E)}(\chi^+,\det)+\dim\Hom_{\Oo(\epsilon V_E)}(\chi^+,\det ), \]
  which is nonzero if and only if $\chi=\chi_F\circ N_{E/F}$ with $\chi_F^2=\omega_{E/F}.$
\end{enumerate}
After  the discussions for the parameter side in \S $4$, we  finish the proof of Theorem \ref{localmain}.
 \section{the Prasad Conjecture for $\SL(2)$}
 Let us recall a well-known result for $\SL_2$.
 \begin{prop}
 	[\cite{shelstad1979notes}] Let $\phi:WD_F\rightarrow \GL_2(\mathbb{C})$ be an irreducible representation and $\tau=i(\phi)=Ad(\phi):WD_F\rightarrow \PGL_2(\mathbb{C})$ be
 	the associated discrete series $L$-parameter for $\SL_2,$ then there is a short exact sequence of component groups
 	\[\xymatrix{1\ar[r]& S_{\phi}\ar[r]& S_{\tau}\ar[r]&I(\phi)\ar[r]&1 } \]
 	where $I(\phi)=\{\chi:F^\times\rightarrow\mathbb{C}^\times|\chi^2=1\mbox{ and }\phi\otimes\chi=\phi \}.$
 \end{prop}
 Assume that $\tau$ is $\SL_2(F)$-distinguished and $\ell\in W_F\backslash W_E,\omega_{E/F}(\ell)=-1.$  We start to verify the Prasad conjecture for $\SL_2$. The main work here is to choose a proper element $A\in\PGL_2(\mathbb{C})$ such that $\tilde{\phi}(\ell)=A$ and $\tilde{\phi}|_{WD_E}=\phi_{\tau}$ for a certain Langlands parameter $\tilde{\phi}\in\Hom(WD_F,\PGL_2(\mathbb{C}))$  under the assumption that $\tau$ is $\SL_2(F)$-distinguished.  According to the discussions in \S3, we separate them into four parts as well.
 	\par Recall that $F^\times/{F^\times}^2\supset\{1,u,\varpi,u\varpi \}$, $E=F(\sqrt{u})$, $E''=F(\sqrt{u\varpi})$  and $E'=F(\sqrt{\varpi} )$. Let $K=F(\sqrt{u},\sqrt{\varpi} )$ be a biquadratic field extension over $F$ with Galois group $\Gal(K/F)=\langle1,s,\sigma,s\sigma \rangle$ and Weil group $W_K$. Suppose that $\Gal(K/E)=\langle s\rangle$, $\Gal(K/E'')=\langle s\sigma\rangle$ and $\Gal(K/E')=\langle\sigma \rangle$. 
 \begin{enumerate}[(A).]
 	\item Assume that $\tau\subset \pi|_{\SL_2(E)}$ is a supercuspidal representation of $\SL_2(E)$.
 	 If  the Langlands parameter of $\tau$  $$\phi_\tau=i(Ind_{W_K }^{W_E}\mu)=\omega_{K/E}\oplus Ind_{W_K}^{W_E}(\frac{\mu^s}{\mu})$$ with $\mu|_{E''}=\chi_F\circ N_{E''/F},$ then $\mu(t)\mu^{s\sigma}(t)=\chi_F(t)$ for $t\in W_K$. So
 	\[\Big(\frac{\mu^s}{\mu} \Big)^\sigma(t)=\frac{\mu^{s\sigma}(t)}{\mu^\sigma(t)}=\frac{\chi_F(t)}{\mu(t)\mu^{\sigma}(t)}=\frac{\chi_F(sts^{-1})}{\mu(t)\mu^\sigma(t) }=\frac{\mu^s(t)}{\mu(t)} \mbox{ for }t\in W_K \]
 	i.e. $\frac{\mu^s}{\mu}=\chi_{E'}\circ N_{K/E'}$ for a character $\chi_{E'}$ of $E'^\times$.
 	\subsection*{Assume $p\neq 2$} 
 	\begin{itemize}
 		\item If $\mu_1^2=\mathbf{1}$, then the Langlands parameter
 		$\phi_{\tau}=\omega_{K/E}\oplus\omega_{K_2/E}\oplus\omega_{K_1/E}, $ where each $K_j\neq K $ is a quadratic field extension over $E$. 
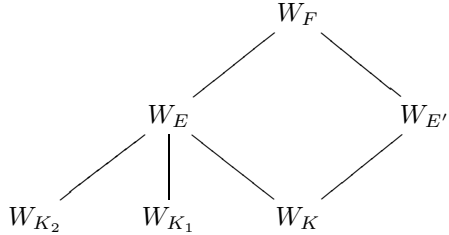
 			\[\xymatrix{&&W_F\ar@{-}[rd]\ar@{-}[ld] &\\& W_E\ar@{-}[ld]\ar@{-}[d]\ar@{-}[rd]& & W_{E'}\ar@{-}[ld]\\ W_{K_2}&W_{K_1}&W_K } \]
 			Set\begin{equation}\label{parameterlift}
 			 \tilde{\phi}=\omega_{E'/F}\oplus Ind_{W_{E'}}^{W_F}\chi_{E'}
 			\end{equation}
 			  where $E'\neq E$ are two distinct quadratic field extensions over $F,$ then
 		$ \tilde{\phi}|_{W_E}=\phi_\tau$.
 		\item If $\mu_1^2\neq\mathbf{1} $, then the Langlands parameter
 		\[\phi_{\tau}=\omega_{K/E}\oplus Ind_{W_K}^{W_E}\frac{\mu^s}{\mu} \]
 		has a lift $\tilde{\phi}$ defined in \eqref{parameterlift}. Moreover,
 		 there is one more lift $$\tilde{\phi}'=\omega_{E'/F}\oplus Ind_{W_{E'}}^{W_F}\chi_{E'}^{-1}\mbox{  with  }\chi_{E'}\circ N_{K/E'}=\frac{\mu^s}{\mu} $$
 		since $Ind_{W_K}^{W_E}(\frac{\mu}{\mu^s})=Ind_{W_K}^{W_E}(\frac{\mu^s}{\mu})$ is irreducible.
 	\end{itemize}
 In the $L$-packet $\Pi_{\phi_\tau}$ containing $\phi_\tau$, set $\phi=Ind_{W_K}^{W_E}\mu$ and $\phi_{\tau}=Ad(\phi)$.
 \par
  If the component group $S_{\phi_{\tau}}$ has order $4$, then we denote the four characters of $S_{\phi_{\tau}}$ by $\{\lambda^{++},\lambda^{--},\lambda^{-+},\lambda^{+-} \}$  which corresponds to the $L$-packet $\Pi_{\phi_\tau}=\{\tau^{++},\tau^{--},\tau^{-+},\tau^{+-} \}$. If the order of $S_{\phi_{\tau}}$  is $2$, then we denote its two characters as $\{\lambda^+,\lambda^- \}$, which corresponds to  $\Pi_{\phi_\tau}=\{\tau^+,\tau^- \}$.
 \begin{itemize}
 	\item If $\mu_1^2=\mathbf{1}$, then $|I(\phi)|=4,$ two of them are $\SL_2(F)$-distinguished and of dimension one, say $\tau^{++}$ and $\tau^{--}$. The component group $S_{\tilde{\phi}}=\mu_2\hookrightarrow S_{\phi_\tau}$ is the diagonal embedding, then $\tau^{+-}$ and $\tau^{-+}$ are not $\SL_2(F)$-distinguished, which is compatible with the fact that neither the restricted representation $\lambda^{+-}|_{S_{\tilde{\phi}} }$ nor $\lambda^{-+}|_{S_{\tilde{\phi} }}$  contains the trivial character of $S_{\tilde{\phi}}$, where $\lambda^{+-}$ (resp. $\lambda^{-+}$) corresponds to the representation $\tau^{+-}$ (resp. $\tau^{-+}$).
 	\item If $\mu_1^2\neq\mathbf{1}$,
 	then $|I(\phi)|=2$ and only one of them is $\SL_2(F)$-distinguished, say $\tau^+=\theta_{\psi,V_K,W }((\frac{\mu^s}{\mu })^+).$  If $\tau^-=\theta_{\psi,\epsilon V_K,W }((\frac{\mu^s}{\mu} )^+)$ corresponds to the nontrivial character of $S_{\phi_\tau},$ denoted by $\lambda^-$, where $\epsilon V_K$ is the $2$-dimensional quadratic space $K$ over $E$ with a quadratic form $\epsilon N_{K/E},~\epsilon\in E^\times\setminus N_{K/E}(K^\times)$ and the Hasse invariant of $Res_{E/F}(\epsilon V_K )$ is $-1,$ then
 	\[\dim \Hom_{\SL_2(F)}(\tau^-,\mathbb{C} )=0. \]
 	Note that $S_{\tilde{\phi}}=\mu_2\cong S_{\phi_\tau }, $ then
 	$\lambda^-|_{S_{\tilde{\phi}} }$ is nontrivial.
 \end{itemize}
 \subsection*{Assume $p=2$}  There are some special cases if $p=2$.
 \begin{itemize}
 	\item 
 If $\dim \Hom_{\SL_2(F)}(\tau,\mathbb{C} )=4,$ then $\mu_1^2=\mathbf{1}$ and there is a quadratic field extension $D$ over $K$ such that $\chi_K=\omega_{D/K}$ and $D$ is the composite field $KE_4$, where $E_4$ is the quadratic field extension of $F$ corresponding to the quadratic character $\frac{\chi'_F}{\chi''_F}$ where $\mu|_{E'}=\chi_F'\circ N_{E'/F},$  $\mu|_{E''}=\chi_F''\circ N_{E''/F}$ and $E',E''$ are two distinct quadratic field extensions over $F,$ which are different from $E.$
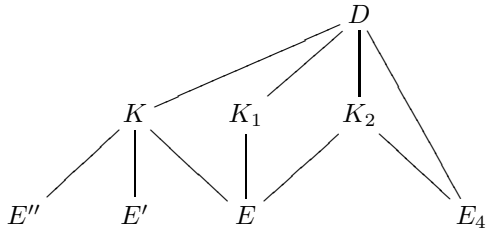
 \[\xymatrix{&&&D\ar@{-}[lld]\\&K\ar@{-}[ld]\ar@{-}[rd]\ar@{-}[d]&K_1\ar@{-}[ru]\ar@{-}[d]&K_2\ar@{-}[rd]\ar@{-}[ld]\ar@{-}[u]\\ E''&E'&E&&E_4\ar@{-}[luu] } \]
 Set $\{1,u,\varpi,d, du,\varpi u,\varpi d,\varpi du \}\subset F^\times/{F^\times}^2$, $K=F(\sqrt{u},\sqrt{\varpi}),~K_2=F(\sqrt{u},\sqrt{d}),~E_4=F(\sqrt{d})$ and $K_1=F(\sqrt{u},\sqrt{d\varpi})$. There are $4$ distinct Langlands parameter lifts of ${\phi}_\tau$ 
 \[\tilde{\phi}_1=\omega_{E_4/F}\oplus\omega_{F(\sqrt{\varpi u})/F}\oplus\omega_{F(\sqrt{d\varpi u})/F },~~\tilde{\phi}_2=\omega_{E_4/F}\oplus \omega_{F(\sqrt{\varpi})/F}\oplus \omega_{F(\sqrt{d\varpi})/F }, \]
 \[\tilde{\phi}_3=\omega_{F(\sqrt{du})/F}\oplus \omega_{F(\sqrt{\varpi u})/F }\oplus\omega_{F(\sqrt{d \varpi})/F }\mbox{ and }\tilde{\phi}_4=\omega_{F(\sqrt{du} )/F }\oplus\omega_{F(\sqrt{\varpi} )/F}\oplus \omega_{F(\sqrt{\varpi ud} )/F }, \]
 where $\omega_{F(\sqrt{\varpi})/F}$ is the quadratic character associated to the quadratic field extension $F(\sqrt{\varpi})/F$, similarly for the other quadratic characters $\omega_{F(\sqrt{du})/F}$ and so on.
 Since $S_{\tilde{\phi}_i}=S_{\phi_{\tau}}\cong \mu_2\times\mu_2 ,$ only $\tau^{++}$ can survive, i.e. the rest elements in the $L$-packet $\Pi_{\phi_\tau}$ can not be $\SL_2(F)$-distinguished. 
 \item
 If $\dim \Hom_{\SL_2(F)}(\tau,\mathbb{C})=1$ and $\pi$ is not dihedral, i.e. $\tau=\pi|_{\SL_2(E)}$ is irreducible, then
 $\phi_{\tau}=\phi^\sigma_\tau$. There exists one element $A\in\PGL_2(\mathbb{C})$ such that 
 \[\phi_{\tau}(\ell\cdot t\cdot\ell^{-1})=A\cdot\phi_{\tau}(t)\cdot A^{-1} \]
 for $t\in WD_E$. Set $\tilde{\phi}(\ell)=A$ and $\tilde{\phi}(t)=\phi_{\tau}(t)$ for $t\in WD_E$. Since $\phi_{\tau}$ is irreducible, $A$ is unique. Hence $\phi_{\tau}$  admits a unique lift $\tilde{\phi}:W_F\rightarrow \PGL_2(\mathbb{C})$ such that $\tilde{\phi}|_{W_E}=\phi_{\tau}.$
\end{itemize}
 	\item If $\phi_\tau(t)=\begin{pmatrix}
 	\chi(t)|t|^z\\&1
 	\end{pmatrix}\in\PGL_2(\mathbb{C}), $ then
 	\begin{itemize}
 		\item if $z=0$ and $\chi$ is trivial, then $\tilde{\phi}(\ell)$ can be chosen as $\begin{pmatrix}
 		\omega_{E/F}(\ell)\\&1
 		\end{pmatrix}=\begin{pmatrix}
 		-1&\\&1
 		\end{pmatrix}$ or $\begin{pmatrix}
 		1\\&1
 		\end{pmatrix};$
 		\item if $z=0,$ $\chi$ does not factor through the norm $N_{E/F}$ but $\chi|_{F^\times}=1,$ set $\chi=\frac{\nu^\sigma}{\nu}$ for a quadratic character $\nu$ of $E^\times$,  then there is only one lift
 		\[\tilde{\phi}=i(Ind_{W_E}^{W_F}\nu );\]
 		\item if $\chi=\chi_F\circ N_{E/F},\chi^2\neq1,$ then there are two lifts
 		\[\tilde{\phi}(\ell)=\begin{pmatrix}
 		\chi_F(\ell)\\&1
 		\end{pmatrix}\mbox{  or  }\begin{pmatrix}
 		-\chi_F(\ell)\\&1
 		\end{pmatrix}. \]
 	\end{itemize}
 	\item If $\phi_\tau=Ad(\mathbf{1}\otimes S_2)$ corresponds to the Steinberg representation $St_E$ of $\SL_2(E)$, then there is only one lift $\tilde{\phi}=Ad(\mathbf{1}\otimes S_2):WD_F\rightarrow\PGL_2(\mathbb{C}).$
 	\item If $\phi_\tau(t)=\begin{pmatrix}
 	\omega_{K/E}(t)\\&1
 	\end{pmatrix}\in\PGL_2(\mathbb{C}),$ then there are several subcases.
 	\begin{itemize}
 		\item If $\omega_{K/E}=\chi_F\circ N_{E/F}$ with $\chi_F^2=\mathbf{1},$ then 
 		\[\tilde{\phi}(\ell)=\begin{pmatrix}
 		\chi_F(\ell)\\&1
 		\end{pmatrix}\mbox{ or }\begin{pmatrix}
 		-\chi_F(\ell)\\&1
 		\end{pmatrix}. \]
 		Moreover, $\omega_{K/E}|_{F^\times}=\chi_F^2=\mathbf{1},$ then $\omega_{K/E}=\frac{\nu^\sigma}{\nu}$ for a quadratic character $\nu$ of $E^\times$, we may set
 		\[\tilde{\phi}_3=i(Ind_{W_E}^{W_F}\nu)=\omega_{E/F}\oplus Ind_{W_E}^{W_F}(\frac{\nu^\sigma}{\nu}). \]
 		\item If $\omega_{K/E}=\chi_F\circ N_{E/F} $ with $\chi_F^2=\omega_{E/F},$ then there is only one extension
 		$$\tilde{\phi}(\ell)=\begin{pmatrix}
 		\chi_F(\ell)\\&1
 		\end{pmatrix} .$$
 		\item If $\omega_{K/E}$ does not factor through the norm map $N_{E/F}$ but $\omega_{K/E}|_{F^\times}=1,$ then
 		\[\tilde{\phi}=i(Ind_{W_E}^{W_F}\nu)\mbox{  where }\omega_{K/E}=\frac{\nu^\sigma}{\nu} .\]
 	\end{itemize}
 \end{enumerate}
Hence, we finish the proof of Theorem \ref{localmain} and Theorem \ref{prasadsl(2)}.
\subsection*{Further discussion}
Inspired by the case that $\tau=\pi|_{\SL_2(E)}$ is an irreducible representation of $\SL_2(E)$, where $\pi$ is a representation of $\GL_2(E)$, we have a certain result of the Prasad conjecture for $G=\Sp_4$. 
\begin{thm}\label{Sp(4)-period} Let $E$ be a quadratic field extension over a nonarchimedean local field $F$ with characteristic zero.
	 Assume that $\tau$ is an irreducible representation of $\Sp_4(E)$.
	Let $\pi$ be an irreducible representation of $\GSp_4(E)$ and $ \pi|_{\Sp_4(E)}\supset\tau$, then
	\begin{enumerate}[(i)]
		\item if $\pi$ is tempered and non-generic, then $\Hom_{\Sp_4(F)}(\tau,\mathbb{C})=0$;
		\item if  $\pi$ is a generic square-integrable representation of $\GSp_4(E)$ 
		 and $\pi|_{\Sp_4(E)}$ is irreducible, then the $L$-packet $\Pi_{\phi_\tau}$ is a singleton and
		\[\dim\Hom_{\Sp_4(F)}(\tau,\mathbb{C} )=|F(\phi_{\tau})| , \]
		where $F(\phi_{\tau})=\{\tilde{\phi}:WD_F\rightarrow \SO_5(\mathbb{C})\big| \tilde{\phi}|_{WD_E}=\phi_{\tau} \}$ and $|F(\phi_{\tau})|$ denotes its cardinality.
	\end{enumerate}
\end{thm}
\begin{proof}
	\begin{enumerate}[(i)]
		\item If $\pi$ is tempered and non-generic, then $\pi=\Theta(\Sigma)$ where $\Sigma$ is an irreducible representation of $\GSO(V_{D_E})$, where $V_{D_E}$ is the nonsplit $4$-dimensional quadratic space over $E$ with trivial discriminant and Hasse invariant $-1$. Since  $Res_{E/F}V_{D_E}$ is an $8$-dimensional quadratic space over $F$ with trivial discriminant  and Hasse invariant $-1$, the conservation relation implies that the theta lift of the trivial representation from $\Sp_4(F)$ to $\Oo(Res_{E/F}V_{D_E})$ is zero.
		Due to the see-saw diagram
		\[\xymatrix{ \tau&\Sp_4(E)\ar@{-}[rd]&\Oo(Res_{E/F}(V_{D_E}))\ar@{-}[ld]&0\\\mathbb{C}&\Sp_4(F)&\Oo(V_{D_E})&\theta(\tau) } \] 
		one has the desired equality $ \Hom_{\Sp_4(F)}(\tau,\mathbb{C} )=0$.
		\item By the assumption, $\tau=\pi|_{\Sp_4(E)}$ is a square-integrable representation. Fix $\ell\in W_F\setminus W_E$.
		\begin{itemize}
			\item If the theta lift $\Theta^{2,2}(\pi)$ from $\GSp_4(E)$ to $\GSO(2,2)(E)$ is zero, then one can use a similar method appearing in the proof of \cite[Theorem 4.2.18(iii)]{lu2016new}
		to obtain the equality	\[\dim\Hom_{\Sp_4(F)}(\pi,\mathbb{C})=\dim\Hom_{\SO(3,3)(F)} (\Theta^{3,3}(\pi),\mathbb{C}), \]
		which is equal to the number
		\[\big|\{\chi:F^\times\rightarrow\mathbb{C}^\times\big| \Hom_{\GSO(3,3)(F)}(\Theta^{3,3}(\pi),\chi\circ\lambda)\neq0  \}\big|, \]
		where $\Theta^{3,3}(\pi)$ is the theta lift of $\pi$ from $\GSp_4(E)$ to $\GSO(3,3)(E)$ and $\lambda$ is the similitude character of the  group $\GSO(3,3)(F).$ Therefore, the dimension $\dim\Hom_{\Sp_4(F)}(\tau,\mathbb{C})=1$ if and only if the Langlands parameter $\phi_{\pi}$ of $\pi$ is conjugate-self-dual, i.e. $\phi_{\pi}^\vee=\phi_{\pi}^\sigma$.
		\par
		On the parameter side, $\phi_{\tau}:WD_E\rightarrow\PGSp_4(\mathbb{C})=\SO_5(\mathbb{C})$ is irreducible and $\phi_{\tau}\cong\phi_{\tau}^\vee\cong\phi_{\tau}^\sigma.$
		There exists a unique element $A\in\SO_5(\mathbb{C})$ such that
		\[\phi_{\tau}(\ell\cdot t\cdot\ell^{-1})=A\cdot\phi_{\tau}(t)\cdot A^{-1} \]
		for $t\in WD_E$. Set $\tilde{\phi}(\ell)=A$ and $\tilde{\phi}(t)=\phi_{\tau}(t)$ for $t\in WD_E$.
	Then $\tilde{\phi}$ is what we want.
		\item If $\Theta^{2,2}(\pi)\neq0,$ 
			 then $\phi_{\pi}=\phi_1\oplus\phi_2$ where $\phi_i:WD_E\rightarrow\GL_2(\mathbb{C})$ is irreducible and $\phi_1\neq\phi_2$. Moreover,  $\phi_\tau=\mathbf{1}\oplus(\phi_1^\vee\otimes\phi_2)$, see \cite[Page 3008]{gan2009sp(4)}.   Let $\Sigma
			 $ be the irreducible representation of $\GSO(2,2)(E)$ satisfying $\theta_\psi(\Sigma)=\pi$, then $\Sigma|_{\SO(2,2)(E)}$ is irreducible since $\pi|_{\Sp_4(E)}$ is irreducible.
			  Using a similar method appearing in \cite[Theorem 4.2.18(ii)]{lu2016new}, one can get that the dimension
			$\dim\Hom_{\Sp_4(F)}(\tau,\mathbb{C})$ has a upper bound 
		\begin{equation}\label{upper}
		\dim\Hom_{\SO(3,3)(F) }(\Theta^{3,3}(\pi),\mathbb{C} )+\dim\Hom_{\SO(4,0)(F)}(\Sigma,\mathbb{C}) 
		\end{equation} 
		and a lower bound
		\begin{equation}\label{lower}
		\sum_X\dim\Hom_{\SO(X,F) }(\Sigma,\mathbb{C} ) 
		\end{equation}
		where $X$ runs over all elements in the kernel $\ker \{H^1(F,\Oo(4))\rightarrow H^1(E,\Oo(4)) \}$. We will show that both the lower bound \eqref{lower} and the upper bound \eqref{upper} are equal to $2$ if $\pi|_{\Sp_4(E)}$ is an irreducible $\Sp_4(F)$-distinguished representation. Then
		\[\dim_\mathbb{C}\Hom_{\Sp_4(F)}(\tau,\mathbb{C}) =2. \]
		 There are two subcases.
		\begin{enumerate}[(a).]
			\item If $\phi_1^\vee=\phi_1^\sigma$, then 
			$\phi_1^\vee\neq\phi_2^\sigma$, otherwise $\phi_1=\phi_2$,
			which contradicts $\phi_1\neq\phi_2$. Since $\phi_1$ is irreducible, the Langlands parameter $\phi_1$ is either conjugate-orthogonal or conjugate-symplectic, but can not be both.  Note that there is an equality
			\[\dim\Hom_{\SO(3,3)(F)}(\Theta^{3,3}(\pi),\mathbb{C})=\big|\{\chi:F^\times\rightarrow\mathbb{C}^\times\big|\Hom_{\GSO(3,3)(F)}(\Theta^{3,3}(\pi),\chi\circ\lambda)\neq0 \} \big|.  \]
		We have a	similar result for both $\dim\Hom_{\SO(4,0)(F) }(\Sigma,\mathbb{C} )$ and $\dim\Hom_{\SO(2,2)(F)} (\Sigma,\mathbb{C} )$.
		If $\phi_2^\vee=\phi_2^\sigma$ is conjugate-self-dual  with the same sign with $\phi_1$, then
		\[\dim\Hom_{\Sp_4(F)}(\tau,\mathbb{C})=2 .\]
		Otherwise, $\tau$ is not $\Sp_4(F)$-distinguished.
		\par
		On the parameter side, $\frac{1}{\det\phi_1}=(\det\phi_1)^\sigma$. Without loss of generality, suppose that $\phi_1$ is conjugate-orthogonal, i.e. $\det\phi_1=\frac{\nu^\sigma}{\nu}=\det\phi_2$, then  $\nu\otimes\phi_j$ is $\Gal(E/F)$-invariant. For each $j$, there exists a parameter $\tilde{\phi}_j:WD_F\rightarrow \GL_2(\mathbb{C})$ such that
		$\tilde{\phi}_j|_{WD_E}=\phi_j\otimes\nu$. Set $\rho_1=\tilde{\phi}_1\oplus\tilde{\phi}_2$ and $\rho_2=\tilde{\phi}_1\oplus\tilde{\phi}_2\omega_{E/F}$. Let $i:\GSp_4(\mathbb{C})\rightarrow\SO_5(\mathbb{C})$ be the natural projection map.
		 Then the parameters
		$i(\rho_1)$ and $i(\rho_2)$ are what we want
			\item If $\phi_1^\vee=\phi_2^\sigma$, then
			$\dim\Hom_{\Sp_4(F)}(\tau,\mathbb{C})=2$ since that the upper bound \eqref{upper} is $2$ and that the lower bound \eqref{lower} is at least $2$. On the parameter side, $\phi_{\tau}=\mathbf{1}\oplus(\phi_2^\sigma\otimes\phi_2 )$ is $\Gal(E/F)$-invariant.
			There exist two natural parameters $\tilde{\phi}_j:WD_F\rightarrow\GL_5(\mathbb{C})$ such that
			$\tilde{\phi}_j|_{WD_E}=\phi_{\tau}$, which are $\omega_{E/F}\oplus As^+(\phi_2)$ and $\omega_{E/F}\oplus As^-(\phi_2)$, where $As^{\pm}(\phi_2)$ are the Asai lifts of $\phi_2$, see \cite[\S7]{gan2011symplectic}. Then the images of $\tilde{\phi}_j$ lie in $\SO_5(\mathbb{C})$.
			
		\end{enumerate}
		\end{itemize}
	\end{enumerate}
Therefore, we have finished the proof.
\end{proof}
\begin{rem}
	If $\tau=\pi|_{\Sp_4(E)}$ is irreducible, one can also use the method appearing in \cite{anandavardhanan2003distinguished} directly to get that the dimension $\dim\Hom_{\Sp_4(F)}(\tau,\mathbb{C})$ equals to the  sum
	\begin{equation}\label{sum} \sum_{\chi:F^\times/(F^\times)^2\rightarrow\mathbb{C}^\times}\dim\Hom_{\GSp_4(F)}(\pi,\chi)。    \end{equation}
	Combining with the results in \cite[Theorem 4.2.18]{lu2016new}, we can obtain  $\dim\Hom_{\Sp_4(F)}(\tau,\mathbb{C})$ if $\pi$ is tempered.
\end{rem}
\begin{rem}  Let $U_2(D)$ be the unique inner form of $\Sp_4(F)$ defined over $F$.
	Suppose that $\pi$ is a generic representation of $\GSp_4(E)$.
	Thanks to \cite[Theorem 1]{beuzart2017distinguished}, if $\pi|_{\Sp_4(E)}=\tau$ is an irreducible square-integrable representation of $\Sp_4(E)$ and $\Theta^{2,2}(\pi)$ is $0$, then
	\[\dim\Hom_{U_2(D)}(\tau,\mathbb{C}) =1 .\]
\end{rem}
\bibliographystyle{amsalpha}
\bibliography{SL(2)}
\end{document}